\documentclass[leqno,english]{amsart}
\usepackage{amsfonts,amssymb,amsmath,amsgen,amsthm}
\usepackage{bbm}
\usepackage{mathrsfs}
\usepackage{hyperref}
\usepackage{xcolor}
\newcommand{\msc}[2][2000]{%
  \let\@oldtitle\@title%
  \gdef\@title{\@oldtitle\footnotetext{#1 \emph{Mathematics subject
        classification.} #2}}%
}

\theoremstyle{plain}
\newtheorem{theorem}{Theorem}[section]

\newtheorem{lemma}[theorem]{Lemma}
\newtheorem{corollary}[theorem]{Corollary}
\newtheorem{proposition}[theorem]{Proposition}

\theoremstyle{remark}
\newtheorem{remark}[theorem]{Remark}

\def\dis
{\displaystyle}

\def\R{{\mathbb R}}
\def\O{\mathcal O}
\def\F{\mathcal F}

\def\En{\mathcal E}

\def\u{\mathbf u}

\def\1{\mathbbm{1}}

\def\({\left(}
\def\){\right)}
\def\<{\left\langle}
\def\>{\right\rangle}
\def\le{\leqslant}
\def\ge{\geqslant}

\def\Eq#1#2{\mathop{\sim}\limits_{#1\rightarrow#2}}
\def\Tend#1#2{\mathop{\longrightarrow}\limits_{#1\rightarrow#2}}

\def\d{{\partial}}
\def\eps{\varepsilon}

\def\si{{\sigma}}

\DeclareMathOperator{\RE}{Re}
\DeclareMathOperator{\IM}{Im}
\DeclareMathOperator{\diver}{div}

\newcommand{\enstq}[2]{\left\{#1~\middle|~#2\right\}}

\numberwithin{equation}{section}

\begin{document}
\title[Dependence of NLS on the power of the nonlinearity]{Dependence of the nonlinear Schr\"odinger flow  
upon the  nonlinearity} 
\author[R. Carles]{R\'emi Carles}
\address{CNRS\\ IRMAR - UMR
  6625\\ F-35000 Rennes, France}
\email{Remi.Carles@math.cnrs.fr}

\author[Q. Chauleur]{Quentin Chauleur}
\address{Univ. Lille, CNRS, Inria, UMR 8524 - Laboratoire Paul Painlevé, F-59000 Lille, France}
\email{quentin.chauleur@inria.fr}

\author[G. Ferriere]{Guillaume Ferriere}
\address{Univ. Lille, CNRS, Inria, UMR 8524 - Laboratoire Paul Painlevé, F-59000 Lille, France}
\email{guillaume.ferriere@inria.fr}

\begin{abstract}
  We consider the defocusing nonlinear
  Schr\"odinger equation in the energy-subcritical case, and
  investigate the dependence of the 
  solution upon the power of the nonlinearity. Special attention is
  paid to the global in time description. The main three aspects
  addressed, in the decreasing order of difficulty, are the limit when
  the total power tends to one, along with the
  connection with the logarithmic Schr\"odinger equation, the
  description when long range effects may be present, and the continuity
  of the scattering operator in the short range case. This text
  resumes the presentation given by the first author at \'Ecole
  polytechnique for the Laurent Schwartz seminar, in May 2026. 
\end{abstract}
\maketitle

\section{Introduction}
\label{sec:intro}

We consider the Cauchy problem associated to the defocusing nonlinear Schr\"odinger equation with power-like nonlinearity,
\begin{equation}
  \label{eq:NLS}
  i\d_t u +\frac{1}{2}\Delta u= |u|^{2\si}u\quad ;\quad u_{\mid t=0}=\phi,
\end{equation}
for $x\in \R^d$, $d \ge 1$, in the  energy-subcritical case, 
$0<\si<\frac{2}{(d-2)_+}$. For clarity, we denote by $\phi_\si$ the
initial data, by $u_\sigma$ the
above solution, as the main goal is to understand the dependence of
$u_\sigma$ with respect to $\sigma$, locally in time and, especially,
globally in time.

Several papers have considered the continuity of the flow map with
respect to the initial data, see
e.g. \cite{CW90,CFH11,Kato95,Kato95corr} for positive results, and 
\cite{KPV01,CCT,CDS12,CaGa25} for lack of continuity (according to the
function space considered). In the present paper, we address the
dependence of the flow map $\phi\mapsto u$ with respect to the
nonlinearity, that is 
with respect to $\si$. Since the nonlinearity is defocusing, the
following quantities are formally conserved by the flow, providing useful
a priori estimates,
\begin{align*}
  &\text{Mass:}\quad \frac{d}{dt}\|u_\si(t)\|_{L^2(\R^d)}^2=0,\\
 & \text{Energy:}\quad \frac{d}{dt}\(\frac{1}{2}\|\nabla
   u_\si(t)\|_{L^2(\R^d)}^2 +
   \frac{1}{\si+1}\|u_\si(t)\|_{L^{2\si+2}(\R^d)}^{2\si+2}\)=0. 
\end{align*}
The
results from \cite{GV79Cauchy} make this statement rigorous: for $\phi_\si\in H^1(\R^d)$, we
have a unique, global, solution $u_\si\in C(\R,H^1)$, and this
solution satisfies the above conservation laws. If in addition
\begin{equation*}
  \phi_\si\in \Sigma = \{f\in H^1(\R^d),\ x\mapsto xf(x)\in L^2(\R^d)\},
\end{equation*}
then $u_\si\in C(\R,\Sigma)$. The questions we address are threefold,
by decreasing order of motivation (and, we believe, of originality):
\begin{enumerate}
\item What can we say about the limit $\si\to 0$? More precisely, up
  to recasting \eqref{eq:NLS}, do we have a convergence to the
  solution of the logarithmic Schr\"odinger equation,
  \begin{equation*}
    i\d_t u+\frac{1}{2}\Delta u = u\ln
    |u|^2?
  \end{equation*}
\item When $\si\le 1/d$, we know that the nonlinear dynamics cannot
  be compared to the linear dynamics in the large time r\'egime (long
  range scattering): can we have a uniform in time continuity of the
  flow map, which would yield some hints on the nature of long range
  effects?
\item When $\si>\si_0(d)$, the Strauss exponent, do we have continuity
  of the nonlinear scattering operator with respect to $\si$? 
\end{enumerate}
We provide rather complete answers to these three questions. The
corresponding statements are given below, as well as a flavor of the
associated proofs. Details can be found in \cite{CCF-p}, on which the
talk is based. 

\section{Preliminary}
\label{sec:prelim}

Resuming the estimates used in order to prove global existence in
$H^1(\R^d)$, based on a fixed point argument relying on Strichartz
estimates, H\"older inequality, and possibly Sobolev embedding (when
$\si>2/d$), it is not difficult to prove:
\begin{proposition}[Local in time continuity]\label{prop:local}
  Let $0<\si<\frac{2}{(d-2)_+}$ and $(\phi_\nu)_{|\nu-\si|\le \eps}$
  bounded in
  $H^1(\R^d)$ for some $\eps>0$ such that $0<\si-\eps$ and $\si+\eps<
  \frac{2}{(d-2)_+}$. Let $T>0$:
  \begin{itemize}
  \item There exists $C$ such that, as $\nu\to \si$,
    \begin{align*}
      &\sup_{t\in [-T,T]}\|u_\nu(t)- u_\si(t)\|_{L^2(\R^d)}\le
        C\|\phi_\nu-\phi_\si\|_{L^2(\R^d)}+ C|\nu-\si|,\\
      & \sup_{t\in [-T,T]}\|u_\nu(t)- u_\si(t)\|_{H^1(\R^d)}\le
        C\|\phi_\nu-\phi_\si\|_{H^1(\R^d)}+C|\nu-\si|^\theta,
    \end{align*}
    for some $\theta\in (0,1]$, which can be taken equal to one if
    $2\si>1$. 
  \item If in addition $(\phi_\nu)_{|\nu-\si|\le \eps}$ is
  bounded in $\Sigma$, then
    \begin{equation*}
      \sup_{t\in [-T,T]}\|u_\nu(t)- u_\si(t)\|_{\Sigma}\le
     C\|\phi_\nu-\phi_\si\|_{\Sigma}+   C|\nu-\si|^\theta,
      \end{equation*}
      for the same $\theta$ as above. 
  \end{itemize}
\end{proposition}
\begin{proof}[Scheme of the proof]
  The function to estimate satisfies
\begin{equation*}
  i\d_t\(u_{\nu}- u_{\sigma}\) +\frac{1}{2}\Delta
  \(u_{\nu}- u_{\sigma}\)=
 \underbrace{|u_{\nu}|^{2\nu}u_{\nu} -
  |u_{\sigma}|^{2\nu}u_{\sigma}}_{\text{"as usual"}}+
 \underbrace{ |u_{\sigma}|^{2\nu}u_{\sigma}
  - |u_{\sigma}|^{2\sigma}u_{\sigma}}_{\text{source term}}.
\end{equation*}
The first group on the right hand side is estimated ``as usual'', with
the technical ingredients recalled above, implying that the difference
$u_\nu-u_\si$ is controlled by the norm of the source term. On the
other hand, Taylor formula yields
\begin{equation*}
  y-y^{1+h}= -hy\ln y \int_0^1 y^{s h} d s,
\end{equation*}
 hence
\[ \left|y-y^{1+h}\right|\lesssim |h| \(y^{1-\eta} + y^{1+\eta}\),\]
for $\eta>0$ arbitrarily small. 
One may think of $y$ as $2\si+1$ homogeneous in $u_\si$, and $h=2(\nu-\si)$
(possibly negative, but always small in absolute value), and the
$L^2$-estimate follows easily. The possible presence of $\theta$ for
the $H^1$ estimate is due typically to the fact that when
differentiating the above equation in space, one faces terms like
\begin{equation*}
  |u_{\sigma}|^{2\nu}\nabla u_{\sigma}
  - |u_{\sigma}|^{2\sigma}\nabla u_{\sigma},
\end{equation*}
and the map $y\mapsto |y|^{2\si}$ is only H\"older continuous when
$2\si<1$. 
\end{proof}

\section{Continuity of the scattering operator}
\label{sec:scatt}

In the case where the power $\si$ is sufficiently large, the previous
continuity can be made uniform in time. The Strauss exponent is given
by
\begin{equation*}
  \si_0(d)= \frac{2-d+\sqrt{d^2+12d+4}}{4d},
\end{equation*}
and we only emphasize the bounds $1/d<\si_0(d)<2/d$.
\begin{theorem}[Global continuity in the (very) short range
  case]\label{theo:global} 
  Let $\si_0(d)<\si<\frac{2}{(d-2)_+}$ and $(\phi_\nu)_{|\nu-\si|\le
    \eps}$  bounded in $\Sigma$ with $\eps>0$ such that
  $\si-\eps>\si_0(d)$ and $\si+\eps<\frac{2}{(d-2)_+}$.
  Then we have the global in time estimate:
  \begin{equation*}
    \sup_{t\in \R}\left\|e^{-i\frac{t}{2}\Delta}\(u_\nu(t)-
      u_\si(t)\)\right\|_{\Sigma}\le
     C\|\phi_\nu-\phi_\si\|_{\Sigma}+   C|\nu-\si|^\theta,
   \end{equation*}
    for the same $\theta$ as in  Proposition~\ref{prop:local}.  
\end{theorem}

We briefly recall the notion of nonlinear scattering operator. For a
given asymptotic state $u_-$, consider the equation \eqref{eq:NLS}
where the initial value is replaced by a final value (at infinite time),
\begin{equation*}
  e^{-i\frac{t}{2}\Delta}u(t)\Big|_{t=-\infty}=u_-. 
\end{equation*}
If this new problem has a solution defined up to $t=0$, then the map
$W_-:u_-\mapsto u_{\mid t=0}$ is called wave operator. Conversely, if the
solution to the Cauchy problem \eqref{eq:NLS} behaves asymptotically
linearly,
\begin{equation*}
  u(t)\Eq t {+\infty} e^{i\frac{t}{2}\Delta}u_+,
\end{equation*}
for some asymptotic state $u_+$, then $u_+=W_+^{-1}u_{\mid t=0}$ and
the scattering operator $S=W_+^{-1}\circ W_-$ is the map $u_-\mapsto
u_+$. This map is well defined on $\Sigma$ for $\si>\si_0(d)$, as
established initially in \cite{GV79Scatt} (see also \cite{CW92,NakanishiOzawa}). In view of the present
context, we emphasize its dependence upon $\si$ by using the notation
$S_\si$. 
\begin{corollary}[Continuity of the scattering
  operator]\label{cor:scattering} 
  Let $\si_0(d)<\si<\frac{2}{(d-2)_+}$ and
  $\(u_{-,_\nu}\)_{|\nu-\si|\le \eps}$ be a family in $\Sigma$ for
  some $\eps>0$. The
  scattering operator is continuous at $\si$ in the sense 
  that if
  \begin{equation*}
    \|u_{-,\si}-u_{-,_\nu}\|_\Sigma\Tend \nu \si 0,
  \end{equation*}
  then
  \begin{equation*}
    \|S_\si u_{-,\si}-S_\nu u_{-,_\nu}\|_\Sigma=\
    \|u_{+,\si}-u_{+,_\nu}\|_\Sigma \Tend \nu \si 0.
  \end{equation*}
\end{corollary}
The extra argument compared to Proposition~\ref{prop:local} relies on
global explicit decay in time of quantities of the form
$\|u_\si(t)\|_{L^p(\R^d)}$, provided by the a priori estimates
stemming from the pseudo-conformal conservation law, discovered in
\cite{GV79Scatt}. 

\section{Uniform in time convergence, including long range cases}
\label{sec:long}

It is well known that as soon as $\si\le 1/d$, the nonlinear dynamics
generated by \eqref{eq:NLS} and the linear dynamics
$e^{i\frac{t}{2}\Delta}$ can be compared only in the case of the
trivial solution (\cite{Barab}):
\begin{equation*}
  \|u_\si(t)- e^{i\frac{t}{2}\Delta}u_+\|_{L^2(\R^d)}\Tend t \infty
    0\Longleftrightarrow \phi_\si=u_\si=u_+=0.
\end{equation*}
In the critical case $\si=1/d$ and for small data, as established
initially in \cite{Ozawa91,GO93,HN98}, the large time behavior of $u$
is described at leading order by a nonlinear phase modification of the
free dynamics. 

This implies for instance that the flow map cannot be
continuous uniformly in time at $\si=1/d$ (as long range effects are
absent for $\si>1/d$, see \cite{BGTV23}). The general consensus is
that even for $\si<1/d$, modified scattering should be characterized
at leading order by a (nonlinear) phase modification, even
though no general proof of this fact seems to be available so far. Our
approach supports this motto: instead of considering $u$,
we first introduce a time dependent rescaling which counterbalances
the linear dispersive behavior, and we consider the squared modulus of this new
function, that is we set
\begin{equation}\label{eq:rho-init}
\rho_\si(t,y)=  \<t\>^d |u_\si(t,y\<t\>)|^2 \|\phi_\si\|_{L^2}^{-2},\quad\text{where}\quad
  \<t\>=\sqrt{1+t^2}, 
\end{equation}
and the last factor ensures, thanks to the conservation of mass,
that $\rho_\si(t,\cdot)$ is a probability density. Our global in time
convergence result involves the Wasserstein (or
Kantorovich–Rubinstein) distance $W_1$, which can be characterized,
for $\mu_1$ and $\mu_2$ probability measures on $\R^d$,  by 
\begin{equation*}
  W_1(\mu_1,\mu_2):= \sup \enstq{\int_{\R^d} \psi d (\mu_1 -
    \mu_2)}{ \psi \in C(\R^d,\R)  \text{ and }  \|\psi\|_{\rm
      Lip}\le 1} , 
\end{equation*}
where the Lipschitz semi-norm is defined by
\begin{equation*}
 \|\psi\|_{\rm Lip} =\sup_{x\not = y}\frac{|\psi(x)-\psi(y)|}{|x-y|}.  
\end{equation*}
See for instance \cite{Vi03}. 
\begin{theorem}[Global continuity for the rescaled
  modulus]\label{theo:W1} 
   Let $0<\si<\frac{2}{(d-2)_+}$ and $(\phi_\nu)_{|\nu-\si|\le \eps}$
  bounded in $\Sigma$ for some $\eps>0$ such that $0<\si-\eps$ and
  $\si+\eps <  \frac{2}{(d-2)_+}$. If 
   \begin{equation*}
     \phi_\nu\Tend \nu \si \phi_\si\quad\text{in }L^2(\R^d),
   \end{equation*}
   then
   \begin{equation*}
     \sup_{t\in \R}W_1\(\rho_\nu(t),\rho_\si(t)\)\Tend \nu \si 0. 
   \end{equation*}
\end{theorem}
\begin{remark}
  As proven in 
\cite[Lemma~2.1]{HaurayMischler2014}, for any
$s>\frac{d+1}{2}$, there exists $C$ such that for any probability
densities $f$ and $g$,
\begin{equation*}
  \|f-g\|_{H^{-s}(\R^d)}\le C W_1(f,g)^{1/2}. 
\end{equation*}
This implies other convergence results than in Theorem~\ref{theo:W1},
involving Sobolev or Lebesgue 
spaces, which we do not detail here, see \cite{CCF-p}.
\end{remark}

\begin{proof}[Informal proof]
  The heuristics for the proof of the above result goes as follows:
  for $T$ arbitrarily large but fixed, we may invoke the local result
  Proposition~\ref{prop:local} on $[-T,T]$. Pretending that instead of
  working with the Wasserstein distance $W_1$, we considered a normed
  space $X$, write, in view of the Fundamental Theorem of Calculus and
  Minkowski inequality,
  for $t>T$,
  \begin{equation*}
    \|\rho_\nu(t)-\rho_\si(t)\|_X\le
    \|\rho_\nu(T)-\rho_\si(T)\|_X+\int_T^t \|\d_t\rho_\nu(s)\|_X ds + 
\int_T^t \|\d_t\rho_\si(s)\|_X ds. 
  \end{equation*}
The normalized density $\rho_\si$ solves a non-autonomous continuity
equation,
\begin{equation*}
  \d_t \rho_\si +\frac{1}{\<t\>^2}\diver j_\si=0,
\end{equation*}
where $j_\si$ is given by 
\[j_\si = \frac{1}{\|\phi_\si\|_{L^2}^2 }\IM\(\overline v_\si\nabla
v_\si\),\]
and $v_\si$ is related to $u_\si$ via the formula
\begin{equation*}
  u_\si(t,x) =
  \frac{1}{\<t\>^{d/2}}v_\si\(t,\frac{x}{\<t\>}\)e^{i\frac{t}{1+t^2}
    \frac{|x|^2}{2}}, 
\end{equation*}
and solves
\begin{equation*}
  i\d_t v_\si +\frac{1}{2\<t\>^2}\Delta v_\si  =
  \frac{|y|^2}{2\<t\>^{2}}v_\si+
  \frac{1}{\<t\>^{d\si}}|v_\si|^{2\si}v_\si\quad ;\quad v_{\si\mid t=0}=\phi_\si.
\end{equation*}
One can prove the estimates
\begin{align*}
&  \|v_\si(t)\|_{L^2}=\|u_\si(t)\|_{L^2}=\|\phi_\si\|_{L^2},\\
&\|\nabla
  v_\si(t)\|_{L^2}\lesssim
    \<t\>^{\max(0,1-d\si/2)},
\end{align*}
and thus, by Cauchy-Schwarz inequality,
\begin{equation*}
  \frac{1}{\<t\>^2}\|j_\si(t)\|_{L^1(\R^d)}\lesssim
  \frac{1}{\<t\>^2}\<t\>^{\max(0,1-d\si/2)}, 
\end{equation*}
which is integrable in time provided that $\si>0$. We refer to
\cite{CCF-p} for the complete proof.
\end{proof}
We note in passing that setting formally $\si=0$ in the above
argument, we lose integrability in time. The limit $\si\to 0$ turns out
to require a totally different approach.

\section{The limit $\si\to 0$}
\label{sec:log}

In the limit $\si\to 0$, at points where $u\not =0$,
\begin{equation*}
  |u|^{2\si}=\exp \(\si \ln|u|^2\) = 1 + \si \ln|u|^2+\O(\si^2). 
\end{equation*}
Up to a gauge transform ($u\to u e^{it}$), we may replace the
nonlinearity in \eqref{eq:NLS} with $\(|u|^{2\si}-1\)u$. We may then
proceed like in \cite{GMS-p} where the stationary, focusing case is
considered, and let $\si\to 0$ in  
\begin{equation}
  \label{eq:rescaledNLS}
    i\d_t \u_\si +\frac{1}{2}\Delta \u_\si = \frac{1}{\si}\(|\u_\si|^{2\si}-1\)\u_\si\quad
    ;\quad \u_{\si\mid t=0}=\phi_\si.
  \end{equation}
Formally, if $\phi_\si\to \phi_0$ as $\si\to 0$,
$\u_\si$ is expected to converge to the solution of the 
logarithmic Schr\"odinger equation
\begin{equation}
  \label{eq:logNLS}
  i\d_t \u_0 +\frac{1}{2}\Delta \u_0 =\u_0\ln(|\u_0|^2)\quad ;\quad
  \u_{0\mid t=0}=\phi_0. 
\end{equation}
This equation was
introduced in \cite{BiMy76}, and the mathematical study of the Cauchy
problem \eqref{eq:logNLS} started in \cite{CaHa80}.
\begin{remark}
  One may adopt an alternative point of view. If $u$ solves \eqref{eq:NLS}, then
  \begin{equation*}
    \tilde u_\si(t,x) = \si^{1/(2\si)} u(t,x) e^{it/\si}
   \end{equation*}
   solves the PDE in \eqref{eq:rescaledNLS}, but with initial value
   $\si^{1/(2\si)} \phi$. The fact that the factor $ \si^{1/(2\si)} $
   is unbounded as $\si\to 0$
   is an alternative evidence that nonlinear effects must be enhanced
   in the nonlinear Schr\"odinger equation in order to get some
   convergence toward the logarithmic Schr\"odinger equation. 
 \end{remark}
Several aspects indicate
that the limit $\si\to 0$ in \eqref{eq:rescaledNLS} is more involved
than  the previous limit $\nu\to \si>0$. Indeed, it was proven in
\cite{CaGa18} that 
solutions to \eqref{eq:logNLS} enjoy properties which are in sharp
contrast with the dynamics associated to \eqref{eq:NLS}. We emphasize
three of them:
\begin{itemize}
\item The dispersion is enhanced by the nonlinearity: the decay is of
  the order $\tau_0(t)^{-d/2}\approx t^{-d/2}\(\ln t\)^{-d/4}$ instead
  of $t^{-d/2}$ in the linear case.
\item Like for the linear heat equation (and as opposed to the
  scattering theory for \eqref{eq:NLS}), there is a universal Gaussian
  profile, if $\phi_0\in \Sigma\setminus\{0\}$,
  \begin{equation*}
    \tau_0(t)^d\left|\u_0\(t,y\tau_0(t)\)\right|^2
    \|\phi_0\|_{L^2(\R^d)}^{-2} \Tend t {+\infty} \Gamma,
  \end{equation*}
  in Wasserstein distance $W_2$, where $\Gamma$ is given by
  \begin{equation*}
    \Gamma(y)=\frac{e^{-|y|^2}}{\pi^{d/2}}. 
  \end{equation*}
  \item The Sobolev norms of any nontrivial solution are unbounded in
    the large 
    time limit, and we have precisely
    \begin{equation*}
      \|u_0(t)\|_{\dot H^s(\R^d)}\Eq t {+\infty} c(s) \(\ln
      t\)^{s/2},\quad \forall s\in [0,1].
    \end{equation*}
\end{itemize}
The local in time convergence $\u_\si\to \u_0$ turns out to be rather
easy, thanks to a property discovered by Cazenave and Haraux
\cite{CaHa80}.

\begin{theorem}\label{theo:log-loc-temps}
  Let $\eps>0$ and $(\phi_\si)_{0\le \si\le \eps}$ bounded in $\Sigma$. Let $\u_\si$ denote the solution to
  \eqref{eq:rescaledNLS}, and $\u_0$ the solution to
  \eqref{eq:logNLS}. There exists $\eps_0>0$ such that for all $\si\in 
  (0,\eps_0)$, the following holds. There exist $C_0,C_1>0$ such that
  for any $T>0$,  
  \begin{equation*}
  \sup_{t\in [-T,T]}  \|\u_\si(t)-\u_0(t)\|_{L^2(\R^d)}\le \(C_1 \si +
  \|\phi_\si-\phi_0\|_{L^2(\R^d)}\) e^{C_0T}.
\end{equation*}
\end{theorem}
The convergence implied by Theorem~\ref{theo:log-loc-temps}
can be understood as a convergence up to some Ehrenfest time, by
analogy with the definition from the semiclassical propagation of
coherent states in (linear) Schr\"odinger equations, see
e.g. \cite{CoRo21}, inasmuch as if
$\|\phi_\si-\phi_0\|_{L^2(\R^d)}=\O(\si)$,  the right hand side goes
to zero as $\si\to 0$, up to  $T= c\ln\frac{1}{\si}$ for any $c\in
(0,1/C_0)$.  
\begin{proof}[Sketch of proof]
  Denote by $w=\u_\si-\u_0$ the error. It solves
\begin{equation}\label{eq:w}
    i\d_t w +\frac{1}{2}\Delta w =
    \frac{1}{\si}\(|\u_\si|^{2\si}-1\)\u_\si - \u_\si\ln(|\u_\si|^2) +
    \u_\si\ln(|\u_\si|^2) -\u_0\ln(|\u_0|^2), 
\end{equation}
with initial value $w_{\mid t=0}=\phi_\si-\phi_0$.  The source term is
\begin{equation}\label{eq:S_si}
  S_\si=  \frac{1}{\si}\(|\u_\si|^{2\si}-1\)\u_\si - \u_\si\ln(|\u_\si|^2) .
\end{equation}
Recall an identity discovered in
\cite{CaHa80}: 
\begin{lemma}[From Lemma~1.1.1 in \cite{CaHa80}]\label{lem:CH}
  There holds
  \begin{equation*}
    \left|\operatorname{Im}\left(\left(z_2 \log
 \left|z_2\right|^2-z_1 \log
 \left|z_1\right|^2\right)\left(\overline{z_2}-
 \overline{z_1}\right)\right)\right| 
    \le 2\left|z_2-z_1\right|^2, \quad \forall z_1, z_2 \in
    \mathbb{C}.
  \end{equation*}
\end{lemma}
  Proceeding as usual to derive $L^2$ estimates in Schr\"odinger
equations, we multiply \eqref{eq:w} by $\overline w$, integrate in
space, and take the imaginary part. This yields, in view of
Cauchy-Schwarz inequality and 
Lemma~\ref{lem:CH},
\begin{equation*}
  \frac{1}{2}\frac{d}{dt}\|w(t)\|_{L^2}^2\le
  \|S_\si\|_{L^2}\|w\|_{L^2} + 2\|w\|_{L^2}^2.
\end{equation*}
The theorem follows from Gr\"onwall lemma, provided that we can show
that $\|S_\si(t)\|_{L^2}=\O(\si)$ on $[-T,T]$. 
  Taylor formula applied to $\si\mapsto
(|z|^{2\si}-1)z=(e^{\si \ln|z|^2}-1)z$  yields
\begin{equation}\label{eq:S_si-Taylor}
  S_\si 
= \si \u_\si \(\ln(|\u_\si|^2)\)^2 \int_0^1 (1-\theta)
    |\u_\si|^{2\theta\si}d\theta, 
\end{equation}
hence the pointwise bound
\begin{equation*}
  |S_\si |\lesssim \si \(\ln(|\u_\si|^2)\)^2 \(
  |\u_\si|+|\u_\si|^{2\si+1}\). 
\end{equation*}
For  $\eta>0$ sufficiently small,  we have the uniform  pointwise
estimate (for $0<\si\le \si_0\ll 1$),
\begin{equation*}
  |S_\si|\lesssim \si \(|\u_\si|^{1-\eta} + |\u_\si|^{1+\eta} \),
\end{equation*}
where the implicit constant depends on $\eta>0$. 
We then use the classical embeddings
\begin{equation*}
  H^1(\R^d)\hookrightarrow L^{2+2\eta}(\R^d),\quad \F
  H^1(\R^d)\hookrightarrow L^{2-2\eta}(\R^d) ,
\end{equation*}
provided that $\eta>0$ is sufficiently small, in terms of $d$. The
theorem then follows from $\Sigma$ bounds for $\u_\si$, stemming from
Lemma~\ref{lem:apriori-unif-si-t}  below.
\end{proof}
Consider the family of ordinary differential equations indexed by
$\si\ge 0$,
\begin{equation}\label{eq:tau_si}
  \ddot\tau_\si = \frac{1}{2\tau_\si^{d\si+1}},\quad
  \tau_\si(0)=1,\quad \dot \tau_\si(0)=0.
\end{equation}
The case $\si=0$ was considered in \cite{CaGa18} to describe the large
time behavior of $\u_0$. For $\si>0$, such equations were introduced
in \cite{CCH22} (up to some scaling described below) in order to
analyze the large time dynamics of 
solutions to equations from compressible fluid mechanics or nonlinear
Schr\"odinger equations. In a first step, we wish to emphasize the
following properties:
\begin{itemize}
\item For $\si>0$, $\dis \tau_\si(t)\Eq t
  {+\infty}\frac{t}{\sqrt\si}$.
\item For $\si=0$, $\dis \tau_0(t)\Eq t {+\infty} t\sqrt{\ln t}$.
\item As proven in \cite{CCF-p}, there exists $C>0$ such that
  \begin{equation*}
    |\tau_\si(t)-\tau_0(t)|\le C \si t \(\ln (t+2)\)^{3/2},\quad
    \forall t\ge 0. 
  \end{equation*}
  In particular, for $\si\approx \frac{1}{\ln t}$, $\tau_\si, \tau_0$
  and the above right hand side have the same order of magnitude,
  suggesting that a transition occurs in this r\'egime.
\end{itemize}
Roughly speaking, the transition between the (long range) scattering
behavior associated to \eqref{eq:rescaledNLS} and the specific
dynamics associated to \eqref{eq:logNLS}, recalled above, is given
mostly by the transition at the level of the ordinary differential
equations, as we will see in Theorem~\ref{theo:log-temps-long}. Before
this final statement and a short description of the arguments of the
proof, we complement the understanding of $\tau_\si$, and provide the
announced uniform estimates for $\u_\si$.
\smallbreak

Consider, for $\alpha>0$, the ordinary differential equation
introduced in \cite{CCH22},
\begin{equation*}
  \ddot r_\alpha =\frac{\alpha}{2r_\alpha^{\alpha+1}},\quad
  r_\alpha(0)=1,\quad \dot r_\alpha(0)=0. 
\end{equation*}
Multiplying by $\dot r_\alpha$ and integrating, we find
\begin{equation*}
  \(\dot r_\alpha\)^2=1-\frac{1}{r_\alpha^\alpha}.
\end{equation*}
This readily shows that $r_\alpha(t)\ge 1$ for all $t\in \R$. Since
$\ddot r_\alpha\ge 0$, $\dot r_\alpha\ge 0$ on $[0,+\infty)$, and
$r_\alpha$ is nondecreasing. If it was bounded, then $\ddot r_\alpha $
would be bounded from below away from zero, hence  a contradiction
after integration. Since $r_\alpha(t)\to +\infty$ as $t\to+\infty$, $\dot
r_\alpha(t) \to 1$, hence
\begin{equation*}
  r_\alpha(t)\Eq t {+\infty} t. 
\end{equation*}
One may be surprised at this stage, as this asymptotic behavior is
independent of $\alpha>0$. We note the identity $r_2(t)=\<t\>$, the
dispersive rate 
considered in \eqref{eq:rho-init} and Theorem~\ref{theo:W1}
\smallbreak

Seeking formally an asymptotic expansion
for $r_\alpha$ of the form $r_\alpha (t) =t +
w_\alpha(t)+\text{h.o.t.}$, we come up with
\begin{equation*}
  2\dot w_\alpha = -\frac{1}{t^\alpha},\text{ hence }w_\alpha(t)
  =\frac{1}{2(\alpha-1)} t^{1-\alpha}.
\end{equation*}
We observe that $w_\alpha(t)$ becomes of the same order as the leading
term $t$ for $\alpha\approx \frac{1}{\ln t}$, meeting the above
remark. Finally, we relate $r_\alpha$ and $\tau_\si$ via the scaling
\begin{equation*}
  \tau_\si(t)=r_{d\si}\(\frac{t}{\sqrt\si}\), 
\end{equation*}
and we note that the limit $\si\to 0$ in \eqref{eq:tau_si} is regular,
while the limit $\alpha\to 0$ for $r_\alpha$ is singular.
\smallbreak

The error  in Theorem~\ref{theo:log-loc-temps} stops being
small for $T= \frac{1}{C_0}\ln \frac{1}{\si}$, that is way before the
transition at the ODE level, which occurs for  $\si\approx
\frac{1}{\ln T}$. This suggests that the analysis requires a different
approach to go up to this kind of time, not to mention infinite time,
which is addressed in our main result, Theorem~\ref{theo:log-temps-long} below.

\smallbreak
Generalizing the change of unknown function introduced in
\cite{CaGa18}, let $v_\si$  given by 
\begin{equation}
  \label{eq:u-v}
  \u_\si(t,x) =
  \frac{1}{\tau_\si(t)^{d/2}}v_\si\(t,\frac{x}{\tau_\si(t)}\)
\exp\(i\frac{\dot\tau_\si(t)}{\tau_\si(t)}\frac{|x|^2}{2}\).
\end{equation}
It solves
\begin{equation*}
  i\d_t v_\si +\frac{1}{2\tau_\si^2}\Delta v_\si
  =\frac{|y|^2}{4\tau_\si^{d\si}}v_\si
  +\frac{1}{\si\tau_\si^{d\si}}|v_\si|^{2\si}v_\si -
    \frac{1}{\si}v_\si. 
\end{equation*}
Up to another gauge transform, we get
\begin{equation}\label{eq:v-sigma}
  i\d_t v_\si +\frac{1}{2\tau_\si^2}\Delta v_\si
  =\frac{|y|^2}{4\tau_\si^{d\si}} v_\si
  +\frac{1}{\si\tau_\si^{d\si}}\(|v_\si|^{2\si}-1\)v_\si . 
\end{equation}
This means that we have replaced the initial $v_\si$ with 
\begin{equation}\label{eq:vtilde-v}
  \tilde v_\si (t,y) = v_\si(t,y) \exp \( -i\frac{t}{\si} +
  i\int_0^t\frac{d s}{\si\tau_\si(s)^{d\si}}\). 
\end{equation}
We have dropped the tildas to lighten notations.
\begin{lemma}\label{lem:apriori-unif-si-t}
Let $\eps\in (0,1/d)$, and $(\phi_\si)_{0\le \si\le \eps}$ bounded in
$\Sigma$. There exists $C$ independent of $t\ge 0$ and $\si\in
(0,\eps]$ such 
  that
\begin{align*}
  \frac{1}{\tau_\si(t)^{2-d\si}}\|\nabla v_\si(t)\|_{L^2}^2 
  + \|yv_\si(t)\|_{L^2}^2
  +
\int_{\R^d}
\left|\frac{|v_\si(t,y)|^{2\si}-1}{\si}\right| \lvert v_\si(t,y)\rvert^2d y\le C,
\end{align*}
and 
\begin{equation*}
  \int_0^\infty \frac{\dot \tau_\si(t)}{\tau_\si(t)^{3-d\si}}\|\nabla
  v_\si(t)\|_{L^2}^2d t\le C. 
\end{equation*}
\end{lemma}
\begin{proof}[Sketch of proof]
  Introduce the
pseudo-energy 
\begin{equation}\label{eq:En-sigma}
  \begin{aligned}
    \En_\si(t)&= \frac{1}{2\tau_\si(t)^2}\|\nabla v_\si(t)\|_{L^2}^2 
+ \frac{1}{4\tau_\si(t)^{d\si}}\|yv_\si(t)\|_{L^2}^2\\
&\quad +
\frac{1}{(\si+1)\tau_\si(t)^{d\si}}\int_{\R^d}
  \(\frac{|v_\si(t,y)|^{2\si}-1}{\si}\)|v_\si(t,y)|^2d y.
  \end{aligned}
\end{equation}
  Since the purely time dependent phase function in \eqref{eq:vtilde-v} does not affect
  $ \mathcal E_\si$, we may consider that $v_\si$ solves
  \eqref{eq:v-sigma}.
  We compute
  \begin{align*}
  \dot\En_\si(t)
  &=-\frac{\dot \tau_\si(t)}{\tau_\si(t)}\Bigg( \frac{1}{\tau_\si(t)^2}\|\nabla v_\si(t)\|_{L^2}^2
   +\frac{d\si}{4\tau_\si(t)^{d\si}} \|y v_\si(t)\|_{L^2}^2 \\
 &\quad
   + \frac{d\si}{(\si+1)\tau_\si(t)^{d\si}}\int_{\R^d}
\(\frac{|v_\si(t,y)|^{2\si}-1}{\si}\)|v_\si(t,y)|^2d y           
\Bigg) .
\end{align*}
We note that $\En_\si$  is not sign-definite, and decompose
it as $\En_\si=\En_\si^+-\En_\si^-$, where 
\begin{align*}
  \En_\si^+(t)&= \frac{1}{2\tau_\si(t)^2}\|\nabla v_\si(t)\|_{L^2}^2 
+ \frac{1}{4\tau_\si(t)^{d\si}}\|yv_\si(t)\|_{L^2}^2\\
&\quad +
\frac{1}{(\si+1)\tau_\si(t)^{d\si}}\int_{|v_\si|\ge 1}
  \(\frac{|v_\si(t,y)|^{2\si}-1}{\si}\)|v_\si(t,y)|^2d y,\\
   \En_\si^-(t)&=\frac{1}{(\si+1)\tau_\si(t)^{d\si}}\int_{|v_\si|< 1}
  \(\frac{1-|v_\si(t,y)|^{2\si}}{\si}\)|v_\si(t,y)|^2d y.
\end{align*}
Now $\En_\si^+$ is the sum of three nonnegative terms, and $\En_\si^-$
is nonnegative. Taylor formula for the function
$f(\si)=y^\si$ yields 
\begin{equation*}
  \frac{1-y^\si}{\si} = \ln \frac{1}{y}\int_0^1 y^{\theta \si}d\theta.
\end{equation*}
We infer
\begin{align*}
  \En_\si^-(t)&\le \frac{1}{(\si+1)\tau_\si(t)^{d\si}}\int_{|v_\si|< 1}
             |v_\si(t,y)|^2\ln \frac{1}{|v_\si(t,y)|^2}d y\\
  &\lesssim \frac{C(\eta)}{(\si+1)\tau_\si(t)^{d\si}}\int_{\R^d}
             |v_\si(t,y)|^{2-2\eta}d y,
\end{align*}
where $\eta>0$ is arbitrarily small. In view of the conservation of
the mass and the embedding $\F H^1\hookrightarrow L^{2-2\eta}$, this implies
\begin{equation}\label{eq:est-En-}
  \En_\si^- (t)\lesssim \frac{1}{\tau_\si(t)^{d\si}}\|y
  v_\si(t)\|_{L^2}^{\frac{d\eta}{2-2\eta}}\lesssim
  \frac{1}{\tau_\si(t)^{d\si}}\(
  \tau_\si^{d\si}\En_\si^+\)^{\frac{d\eta}{1-1\eta}} .
\end{equation}
We infer in particular that $\En_\si(0)$ is bounded uniformly in
$\si\in (0,1/d)$. 
In view of the derivative of $\En_\si$, we also have
\begin{equation*}
  \frac{d}{d t}\(\tau_\si^{d\si}\En_\si\) = -\dot\tau_\si
  \tau^{d\si-1}\(1-\frac{d\si}{2}\) \frac{1}{\tau_\si^2}\|\nabla
  v_\si\|_{L^2}^2\le -\dot\tau_\si
  \tau_\si^{d\si-1} \frac{1}{2\tau_\si^2}\|\nabla
  v_\si\|_{L^2}^2,
\end{equation*}
where we have used the facts that $\tau_\si,\dot\tau_\si\ge 0$ and
$\si<1/d$. We obtain the uniform bound
\begin{equation*}
  \En_\si(t)\le \frac{\En_\si(0)}{\tau_\si(t)^{d\si}}\le
  \frac{C}{\tau_\si(t)^{d\si}}. 
\end{equation*}
Invoking \eqref{eq:est-En-}, we have
\begin{equation*}
  \tau_\si^{d\si}\En_\si^+\le C + \tau_\si^{d\si}\En_\si^- \lesssim 1
  + \(  \tau_\si^{d\si}\En_\si^+\)^{\frac{d\eta}{1-\eta}}.
\end{equation*}
Taking $\si>0$ such that $\frac{d\eta}{1-\eta}<1$ shows that
$\tau_\si^{d\si}\En_\si^+$ is bounded uniformly in $t\ge 0$ and
$\si\in (0,1/d)$. Again from \eqref{eq:est-En-}, this implies that so
is $\tau_\si^{d\si}\En_\si^-$, hence
$\tau_\si^{d\si}\(\En_\si^++\En_\si^-\)\le C$, which is the first
claim of the lemma. We infer that $\tau_\si^{d\si}\En_\si$ is uniformly bounded from
below, so its derivative is integrable,
\begin{equation*}
   \int_0^\infty
   \frac{\dot\tau_\si(t)}{\tau_\si(t)^{3-d\si}}\|\nabla
  v_\si(t)\|_{L^2}^2d t\le C,
\end{equation*}
which completes the proof. 
\end{proof}

Our final result is the following uniform in time convergence:
\begin{theorem}\label{theo:log-temps-long}
 Let $\eps>0$, $(\phi_\si)_{0\le \si\le \eps}$ bounded in $\Sigma$,
with $\phi_\si\to \phi_0$ in $L^2(\R^d)$ as $\si\to 0$. Define, for
$0\le \si\le \eps$
\begin{equation*}
    \varrho_\si(t,y) = \tau_\si(t)^d\left\lvert
    \u_\si\(t,y\tau_\si(t)\)\right\rvert^2 \|\phi_\si\|_{L^2}^{-2}.
\end{equation*}
We have the uniform in time convergence in Wasserstein distance:
\begin{equation*} 
  \sup_{t\ge 0}W_1\(\varrho_\si(t),\varrho_0(t) \) \Tend \si 0 0 . 
\end{equation*}
Moreover, if $\| \phi_{\sigma} - \phi_0 \|_{L^2(\R^d)}=\mathcal{O}(\sigma)$, we have the convergence rate
\begin{equation*} 
  \sup_{t\ge 0}W_1\(\varrho_\si(t),\varrho_0(t) \) \lesssim  \frac{1}{ \sqrt{\ln \ln \frac{1}{\sigma}}}.
\end{equation*}
\end{theorem}
We note that the convergence of $\varrho_\si$ toward $\varrho_0$ holds
uniformly in time, and so supersedes the range  on which $\tau_\si$
converges to $\tau_0$.
\smallbreak

We only describe some steps of the proof and evoke the main technical
ingredients of the rather long proof given in \cite{CCF-p}. First, we
consider the Madelung transform associated to $v_\si$:
\begin{equation*}
  \rho_\si = |v_\si|^2,\quad J_\si =\IM\(\overline v_\si\nabla
  v_\si\). 
\end{equation*}
The hydrodynamical unknowns solve
\begin{equation*}
  \left\{
    \begin{aligned}
      &\d_t \rho_\si +\frac{1}{\tau_\si^2}\diver J_\si=0,\\
      &\d_t J_\si +\frac{1}{(\si+1)\tau_\si^{d\si}}\nabla
        \rho_\si^{\si+1} +\frac{2}{\tau_\si^{d\si}} y\rho_\si =
        \frac{1}{4\tau_\si^2}\Delta \nabla \rho_\si -
        \frac{1}{\tau_\si^2}\diver \nu_\si,
    \end{aligned}
  \right.
\end{equation*}
where $\nu_\si = \RE\(\nabla v_\si\otimes \nabla \overline
v_\si\)$. We eliminate the momentum $J_\si$ by considering $\d_t\(
\tau_\si^2\d_t \rho_\si\)$:
\begin{equation*}
  \d_t\(\tau_\si^2\d_t \rho_\si\) = \frac{1}{(\si +1) \tau_{\si}^{d\si}} \Delta \rho_{\si}^{ \si +1} + \frac{2}{\tau_{\si}^{d\si}} \diver (y \rho_{\si}) - \frac{1}{4 \tau_{\si}^2} \Delta^2 \rho_{\si} + \frac{1}{\tau_{\si}^2} \diver \diver \nu_{\si}.
\end{equation*}
We develop
\begin{equation*}
  \d_t\(\tau_\si^2\d_t \rho_\si\) = 2\tau_\si\dot \tau_\si\d_t
  \rho_\si + \tau_\si^2\d_t^2 \rho_\si,
\end{equation*}
and recall that in the case $\si=0$, the term $\tau_\si^2\d_t^2
\rho_\si$ is negligible in the large time limit (see \cite{CaGa18} for
an argument based on compactness, and \cite{Ferriere2021} for a
quantitative proof). Taking into account the time dependent factor in
front of the first two terms on the right hand side of the equation
for $\d_t\(\tau_\si^2\d_t \rho_\si\) $, we change the time variable
$t$ to $s$, in order to have
\begin{equation*}
  \frac{\d}{\d s} = (1-d\si) \dot\tau_\si(t) \tau_\si(t)^{d\si+1}
  \frac{\d}{\d t}.
\end{equation*}
The factor $1-d\si$ is somehow cosmetic, to shorten formulas in
\cite{CCF-p}. Recalling \eqref{eq:tau_si}, $\dot\tau_\si(t)
\tau_\si(t)^{d\si+1}= \dot\tau_\si(t)/(2\ddot \tau_\si(t))$, so we
compute ``explicitly''
\begin{equation*}
  s=s_\si(t) = \frac{1}{2(1-\d\si)}\ln \dot \tau_\si(t) =
  \frac{1}{4(1-d\si)}
  \ln\(\frac{1}{d\si}\(1-\frac{1}{\tau_\si(t)^{d\si}}\)\). 
  \end{equation*}
  For fixed $\si>0$, the time variable has been compactified, as
  \begin{equation*}
    s_\si(t)\Tend t {+\infty} \frac{1}{4(1-d\si)}
  \ln\(\frac{1}{d\si}\), 
\end{equation*}
but this compact becomes unbounded in the limit $\si\to 0$.
The new unknown function $\tilde \rho_\si$ given by
\begin{equation*}
  \tilde \rho_\si(s_\si(t),y) = \rho_\si(t,y)
\end{equation*}
solves an equation of the form
\begin{equation*}
  \d_s \tilde \rho_\si = \frac{1}{\si+1}\Delta \tilde \rho_\si^{\si+1}
  +2\diver\(y\tilde \rho_\si\) +\mathcal R,
\end{equation*}
where $\mathcal R$ is expected to be negligible in the limit $\si\to
0$, for sufficiently large time. Making this loose statement rigorous
turns out to be the core of the proof of
Theorem~\ref{theo:log-temps-long}. The connection to the large time
behavior of $\u_0$ can then be guessed as follows: forgetting
$\mathcal R$,  we get the porous medium equation,
\begin{equation*}
  \d_s f = \frac{1}{\si+1}\Delta f^{\si+1} +2\diver(y f),
\end{equation*}
for which it was proven initially in \cite{Otto2001} that the solution
converges, in the large time limit, to a Barenblatt profile, in
Wasserstein distance. On the other hand, as $\si\to 0$, this
Barenblatt profile converges to the Gaussian $\Gamma$ (up to suitable
renormalization), see e.g. \cite{Chauleur2022}. The proof in
\cite{CCF-p} consists indeed in measuring the distance between
$\tilde \rho_\si$ and $\Gamma$, by using the following tools and
properties:
\begin{itemize}
\item Formally setting $\si=0$, the right hand side of the porous
  medium equation involves the harmonic Fokker-Planck operator $L$,
  \begin{equation*}
    Lf= \Delta f +2 \diver(yf). 
  \end{equation*}
\item Under suitable assumptions on $\phi$, the following convergence is
  classical (see e.g. \cite{AmbrosioGigliSavare}),
  \begin{equation*}
    W_1\( e^{sL} \phi,\Gamma\)\le W_2\( e^{sL} \phi,\Gamma\)\le
    e^{-2s} W_2\( \phi,\Gamma\).
  \end{equation*}
\item In view of \cite{Ferriere2021},
  \begin{equation*}
    W_1\( \varrho_0(t),\Gamma\)\lesssim \frac{1}{\sqrt{\ln t}}.
  \end{equation*}
  This explains the rate announced in
  Theorem~\ref{theo:log-temps-long}, where, among others, we match the error estimate
  from Theorem~\ref{theo:log-loc-temps} with this estimate, typically at
  $T=\frac{1}{2C_0}\ln\frac{1}{\si}$. 
 \item The duality and regularizing techniques introduced in
   \cite{Ferriere2021} for the case $\si=0$ are resumed to deal with
   the case $0<\si\ll 1$.
 \item We use fine estimates for $\tau_\si$, $\tau_\si-\tau_0$ and
   $v_\si$. 
\end{itemize}
\bibliographystyle{abbrv}
\bibliography{../continuity}

\end{document}